\documentclass{amsart}

\DeclareMathOperator{\Hom}{Hom}

\DeclareMathOperator{\Rep}{Rep}

\usepackage{amscd}
\usepackage{amssymb}

\numberwithin{equation}{subsection}

\newtheorem{theorem}{Theorem}[subsection]
\newtheorem{corollary}[theorem]{Corollary}
\newtheorem{lemma}[theorem]{Lemma}
\newtheorem{proposition}[theorem]{Proposition}

\theoremstyle{definition}

\begin{document} 

\title[Comodules for some simple $\mathcal O$-models for $\mathbb G_m$]
{Comodules for some simple $\mathcal O$-forms of $\mathbb G_m$}

\author{N.\ E.\ Csima}
\address{N.\ E.\ Csima\\Department of Mathematics\\ University of Chicago\\
5734 University Avenue\\ Chicago, Illinois 60637}
\curraddr{}
\email{ecsima@math.uchicago.edu} 

\author{R.\ E.\ Kottwitz}
\address{Robert E. Kottwitz\\Department of Mathematics\\ University of Chicago\\ 5734 University
Avenue\\ Chicago, Illinois 60637}

\email{kottwitz@math.uchicago.edu}
\thanks{Partially supported by NSF Grant DMS-0245639}


\begin{abstract} 
This paper gives a rather concrete description of the category $\Rep(G)$ for
certain flat commutative affine group schemes $G$ over a discrete valuation
ring such that the general fiber of $G$ is the multiplicative group. 
\end{abstract}

\maketitle

Tannakian theory allows one to understand an affine group scheme $G$ over a
commutative base ring $A$ in terms of the category $\Rep(G)$ of $G$-modules,
by which is meant comodules for the Hopf algebra corresponding to $G$. The
theory is especially well-developed \cite{saavedra} in the case that $A$ is
a field, and some parts of the theory still work well over more general
rings $A$, say discrete valuation rings (see \cite{saavedra,wedhorn}). 

When $A$ is a field of characteristic zero and $G$ is connected reductive,
the category $\Rep(G)$ is very well understood. However, with the exception
of groups as simple as the multiplicative and additive groups, little seems
to be known about what $\Rep(G)$ looks like concretely when $A$ is no longer
assumed to be a field, even in the most favorable case in which $A$ is a
discrete valuation ring and $G$ is a flat affine group scheme over $A$ with
connected reductive general fiber. 

The modest goal of this paper is to give a  concrete description of
$\Rep(G)$ for certain flat group schemes $G$ over a discrete valuation ring
$\mathcal O$ such that the general fiber of $G$ is $\mathbb G_m$. 

Choose a generator $\pi$ of the maximal ideal of $\mathcal O$ and write $F$
for the field of fractions of $\mathcal O$. For any non-negative integer $k$,
the construction of \ref{sec.gs}, when applied to $f=\pi^k$, yields a
commutative flat affine group scheme $G_k$ over $\mathcal O$ whose general
fiber is $\mathbb G_m$. The $\mathcal O$-points of $G_k$ are given by 
\[
G_k(\mathcal O)=\{t \in \mathcal O^\times\,:\,t\equiv1 \mod{\pi^k}\}.
\]
These form a projective system 
\[
\dots \to G_2 \to G_1 \to G_0=\mathbb G_m
\]
in an obvious way, and we may form the projective limit $G_\infty:=\projlim
G_k$. The Hopf algebra $S_k$ corresponding to $G_k$ can be described
explicitly (see \ref{sec.gs} and \ref{sec.sbr}). The Hopf algebra $S_\infty$
corresponding to $G_\infty$ is 
\[
\injlim S_k=\{\sum_{i \in \mathbb Z} x_iT^i \in F[T,T^{-1}]\,:\, \sum_{i \in
\mathbb Z} x_i \in \mathcal O\}.
\] 

The categories $\Rep(G_\infty)$ and $\Rep(G_k)$ can be described very
concretely. Indeed, $\Rep(G_\infty)$ consists of the category of $\mathcal
O$-modules $M$ equipped with a $\mathbb Z$-grading on $F\otimes_\mathcal O
M$ (see \ref{sec.dvr}, where a much more general result is proved). As for
$\Rep(G_k)$, we proceed in two steps.

First, the full subcategory of $\Rep(G_k)$ consisting of those $G_k$-modules
that are flat as $\mathcal O$-modules is equivalent (see Theorem
\ref{thm.main}) to the category of pairs $(V,M)$ consisting of a $\mathbb
Z$-graded $F$-vector space $V$ and an \emph{admissible} $\mathcal
O$-submodule
$M$ of $V$, where admissible means that the canonical map $F\otimes_\mathcal
O M \to V$ is an isomorphism and $C_nM \subset M$ for all $n\ge 0$, where
$C_n:V\to V$ is the graded linear map given by multiplication by
$\pi^{kn}\binom{i}{n}$ on the $i$-th graded piece of $V$. The $G_k$-module
corresponding to $(V,M)$ is $M$, equipped with the obvious comultiplication. 

Second, any $G_k$-module (see \ref{sec.pid}) is obtained as the cokernel of
some injective homomorphism $M_1 \to M_0$ coming from a morphism $(V_1,M_1)
\to (V_0,M_0)$ of pairs of the type just described.

When $\mathcal O$ is a $\mathbb Q$-algebra, the situation is even simpler:
$M$ is an admissible $\mathcal O$-submodule of the graded vector space $V$
if and only if $C_1M \subset M$ and $F\otimes_\mathcal O M \cong V$.
Moreover, in case $\mathcal O$ is the formal power series ring $\mathbb
C[[\varepsilon]]$, there is an interesting connection with affine Springer
fibers (see \ref{sec.asf}).

\section{A description of $\Rep(G)_f$ for certain group schemes $G$}  
Throughout this section we consider a commutative ring $A$ and a
nonzerodivisor $f \in A$. Thus the canonical homomorphism $A \to A_f$ is
injective, where $A_f$ denotes the localization of $A$ with respect to the
multiplicative subset $\{f^n:n\ge0\}$. For the rest of this section we
denote $A_f$ by $B$ and use the canonical injection $A \hookrightarrow B$ to
identify $A$ with a subring of $B$. 

\subsection{The group scheme $G$ over $A$} \label{sec.gs}
We are now going to define a
commutative affine group scheme $G$, flat and finitely presented over $A$.
There will be a canonical homomorphism $G \to \mathbb G_m$ that becomes an
isomorphism after extending scalars from $A$ to $B$. 

We begin by specifying the functor of points for $G$. For any commutative
$A$-algebra $R$ we put 
\begin{align*}
G(R):&=\{(t,x) \in R^\times \times R: t-1=fx\} \\ 
&=\{x \in R: 1+fx \in R^\times\}.
\end{align*}
Then $G$ is represented by the $A$-algebra 
\begin{align}
\label{eq.dag} S:&=A[T,T^{-1},X]/(T-1-fX)\\ 
&=A[X]_{1+fX}, \notag
\end{align} 
which is clearly flat and finitely presented. 

The multiplication on $G(R)$ is defined as $(t,x)(t',x')=(tt',x+x'+fxx')$.
The identity element is $(1,0)$ and the inverse of $(t,x)$ is
$(t^{-1},-t^{-1}x)$. 

There is a canonical homomorphism $\lambda:G \to \mathbb G_m$, given by
$(t,x)
\mapsto t$. When $f$ is a nonzerodivisor in $R$, the
homomorphism $\lambda:G(R) \to R^\times$ identifies $G(R)$ with
$\ker[R^\times
\to (R/fR)^\times]$, and when $f$ is a unit in $R$, then $G(R)=R^\times$,
showing that the homomorphism $\lambda:G \to \mathbb G_m$ becomes an
isomorphism after extending scalars from $A$ to $B$. Thus there is a
canonical isomorphism $B\otimes_A S\cong B[T,T^{-1}]$. 

\begin{lemma}\label{lem.flat}
Let $M$ be an $A$-module  on which $f$ is a
nonzerodivisor. Let $F$ be any flat $A$-module. Then  $f$ is also a
nonzerodivisor on
$F\otimes_A M$. 
\end{lemma} 
\begin{proof}
Tensor the injection $M \xrightarrow{f} M$ over $A$ with $F$. 
\end{proof} 
\begin{corollary}
The canonical homomorphism $S \to B\otimes_A S=B[T,T^{-1}]$ is injective, so
that we may identify $S$ with a subring of $B[T,T^{-1}]$. 
\end{corollary}
\begin{proof}
Just note that $S$ is flat over $A$ and $f$ is a nonzerodivisor on $A$.
Therefore $f$ is a nonzerodivisor on $S\otimes_A A=S$, and this means that
$S \to B \otimes_A S$ is injective. 
\end{proof} 

\subsection{Description of $S$ as a subring of $B[T,T^{-1}]$}\label{sec.sbr} 
 We have just identified $S$ with a subring
of $B[T,T^{-1}]$. It is obvious from
\eqref{eq.dag} that $S$ is the $A$-subalgebra of $B[T,T^{-1}]$ generated by
$T,T^{-1},\frac{T-1}{f}$. However there is a more useful description of $S$
in terms of $B$-module maps 
\[
L_n:B[T,T^{-1}] \to B,
\] 
one for each non-negative integer $n$, defined by the formula 
\[
L_n\bigl (\sum_{i \in \mathbb Z} b_iT^i\bigr)=\sum_{i \in \mathbb
Z}f^n\binom{i}{n}b_i.
\] 
Here $\binom{i}{n}$ is the binomial coefficient $i(i-1)\dots(i-n+1)/n!$,
defined for all $i \in \mathbb Z$. When $n=0$, we have $\binom{i}{n}=1$ for
all $i \in \mathbb Z$. 

The following remarks may help in understanding the maps $L_n$. For any
non-negative integer $n$, we have the divided-power differential operator
\[
D^{[n]}:B[T,T^{-1}] \to B[T,T^{-1}]
\] defined by 
\begin{equation}
D^{[n]}\bigl( \sum_{i \in \mathbb Z} b_iT^i\bigr)=\sum_{i\in\mathbb
Z}\binom{i}{n}b_iT^{i-n}.
\end{equation}
The Leibniz formula says that 
\begin{equation}\label{eq.dot}
D^{[n]}(gh)=\sum_{r=0}^n D^{[r]}(g)D^{[n-r]}(h). 
\end{equation}
For any $g \in B[T] \subset B[T,T^{-1}]$ the Taylor expansion of $g$ at
$T=1$ reads 
\begin{equation}\label{eq.ddot}
g=\sum_{n=0}^\infty \bigr(D^{[n]}g\bigr)(1)\cdot(T-1)^n,
\end{equation}
the sum having only finitely many non-zero terms. 

For any $g \in B[T,T^{-1}]$ we have $L_n(g)=f^n(D^{[n]}g)(1)$.  It follows
from \eqref{eq.dot} that for all $g,h \in B[T,T^{-1}]$ 
\begin{equation}\label{eq.dot'}
L_n(gh)=\sum_{r=0}^n L_r(g)L_{n-r}(h), 
\end{equation}
and for all $h \in B[T] \subset B[T,T^{-1}]$ it follows from \eqref{eq.ddot}
that 
\begin{equation}\label{eq.ddot'}
h=\sum_{n=0}^\infty L_n(h)\biggl(\frac{T-1}{f}\biggr)^n.
\end{equation} 

Now we are in a position to prove 
\begin{proposition}\label{prop.des}
The subring $S$ of $B[T,T^{-1}]$ is equal to 
\[
\{g \in B[T,T^{-1}]\,:\,\forall n \ge 0 \,\,\,\, L_n(g) \in A\}.
\]
\end{proposition}
\begin{proof}
Write $S'$ for $\{g \in B[T,T^{-1}]\,:\,\forall n \ge 0 \,\,\,\, L_n(g) \in
A\}$. Obviously $S'$ is an $A$-submodule of $B[T,T^{-1}]$, and it follows
from \eqref{eq.dot'} that $S'$ is a subring of $B[T,T^{-1}]$. A simple
calculation shows that $T,T^{-1},(T-1)/f$  lie in $S'$, and as these
three elements generate $S$ as $A$-algebra, we conclude that $S \subset
S'$. 

Now let $g \in S'$. There exists an integer $n$ large enough that $h:=T^mg$
lies in the subring $B[T]$. Note that $h \in S'$. Equation \eqref{eq.ddot'}
shows that
$h
\in S$, since $(T-1)/f \in S$ and $L_n(h) \in A$.  
Therefore $g=T^{-m}h \in S$.
\end{proof} 

Now let $M$ be an $A$-module on which $f$ is a nonzerodivisor, so that we
may use the canonical $A$-module map $M \to B\otimes_A M$ (sending $m$ to
$1\otimes m$) to identify $M$ with an $A$-submodule of $N:=B\otimes_A M$. 

It follows from Lemma \ref{lem.flat} that the canonical $A$-module map 
\[
S\otimes_A M \to B \otimes_A(S\otimes_AM)=B[T,T^{-1}]\otimes_B N
\] 
identifies $S \otimes_A M$ with an $A$-submodule of $B[T,T^{-1}]\otimes_B
N$. We will now derive from the previous proposition a description of
$S\otimes_A M$ inside $B[T,T^{-1}]\otimes_B N$. For this we will need the
$B$-module maps $\mathbf L_n:B[T,T^{-1}]\otimes_B N \to N$ defined by 
\[
\mathbf L_n\bigl(\sum_{i \in \mathbb Z} T^i \otimes x_i \bigr)=\sum_{i \in
\mathbb Z} f^n\binom{i}{n}x_i.
\] 
Here $x_i \in N$, all but finitely many being $0$. 
\begin{lemma}\label{lem.sm}
The $A$-submodule  $S\otimes_A M$ of $B[T,T^{-1}]\otimes_B N$ is equal to 
\[
\{x \in B[T,T^{-1}]\otimes_B N \,:\,\forall n \ge 0 \,\,\,\, \mathbf L_n(x)
\in M\}.
\]
\end{lemma} 
\begin{proof}
From Proposition \ref{prop.des} we see that there is an exact sequence 
\[
0 \to S \to B[T,T^{-1}] \xrightarrow{L}\prod_{n \ge 0} B/A,
\] 
the $n$-th component of the map $L$ being the composition
\[
B[T,T^{-1}]\xrightarrow{L_n} B \twoheadrightarrow B/A. 
\]
In fact the map $L$ takes values in $\oplus_{n\ge 0} B/A$. Indeed, for any
$g \in B[T,T^{-1}]$ there exists an integer $m$ large enough that $f^mg \in
A[T,T^{-1}]$, and then $L_n(g) \in A$ for all $n\ge m$. Moreover $L$ maps
$B[T,T^{-1}]$ \emph{onto} $\oplus_{n\ge 0} B/A$. Indeed, a simple
calculation shows that for $b \in B$ and $m\ge 0$ 
\[
L_n\bigl(bf^{-m}(T-1)^m\bigr)=
\begin{cases}
b \quad \text{if $m=n$}\\
0 \quad  \text{otherwise}. 
\end{cases}
\]
(First check that $D^{[n]}\bigl((T-1)^m\bigr)=\binom{m}{n}(T-1)^{m-n}$, say
by induction on $m$; note that this formula is valid even if $n>m$, since
$\binom{m}{n}=0$ when $0 \le m <n$.)

We now have a short exact sequence 
\[
0 \to S \to B[T,T^{-1}] \xrightarrow{L}\bigoplus_{n \ge 0} B/A \to 0
\] 
of $A$-modules. Tensoring with the $A$-module $M$, we obtain an exact
sequence 
\begin{equation}\label{eq.es}
S\otimes_A M \to B[T,T^{-1}]\otimes_A M \xrightarrow{L\otimes{id_M}}
\bigl(\bigoplus_{n \ge 0} B/A \bigr)\otimes_A M \to 0.
\end{equation}
Now \[
B[T,T^{-1}]\otimes_A M=B[T,T^{-1}]\otimes_B B \otimes_A
M=B[T,T^{-1}]\otimes_B N
\] 
and 
\[
\bigl(\bigoplus_{n \ge 0} B/A \bigr)\otimes_A
M =\bigoplus_{n \ge0} N/M.
\]
 With these identifications (and
recalling that $S\otimes_A M \to B[T,T^{-1}]\otimes_B N$ is injective), we
see that \eqref{eq.es} describes $S\otimes_A M$ as the subset of
$B[T,T^{-1}]\otimes_B N$ consisting of elements $x$ such that $\mathbf
L_n(x) \in M$ for all $n\ge 0$, and this completes the proof. 
\end{proof} 

\subsection{Comodules for $S$} Since $G$ is an affine group scheme over $A$,
the $A$-algebra $S$ is actually a commutative Hopf algebra, and we can
consider $\Rep(G)$, the category of $S$-comodules. We denote by $\Rep(G)_f$
the full subcategory of $\Rep(G)$ consisting of $S$-comodules $M$ such that
$f$ is a nonzerodivisor on the $A$-module underlying $M$. Our next goal is
to give a concrete description of $\Rep(G)_f$. 

In order to do so, we need one more construction. Let $N=\oplus_{i \in
\mathbb Z}N_i$ be a $\mathbb Z$-graded $B$-module. For each non-negative
integer $n$ we define an endomorphism $C_n:N \to N$ of the graded
$B$-module $N$ by requiring that $C_n$ be given by  multiplication by
$f^n\binom{i}{n}$ on $N_i$. Thus 
\[
C_n\bigl(\sum_{i\in \mathbb Z}x_i\bigr)=\sum_{i \in \mathbb
Z}f^n\binom{i}{n}x_i.
\] 
Here $x_i \in N_i$, all but finitely many being $0$. 

Let $\mathcal C$ be the category whose objects are pairs $(N,M)$, $N$ being
a $\mathbb Z$-graded $B$-module, and $M$ being an $A$-submodule of $N$ such
that the natural map $B\otimes_A M \to N$ is an isomorphism and such that
$C_nM \subset M$ for all $n \ge 0$. A morphism $(N,M) \to (N',M')$ is a
homomorphism $\phi:N \to N'$ of graded $B$-modules such that $\phi M \subset
M'$. 

We now define a functor $F:\Rep(G)_f \to \mathcal C$. Let $M$ be an object
of $\Rep(G)_f$. Then $N:=B\otimes_A M$ is a comodule for $B\otimes_A
S=B[T,T^{-1}]$. It is known (see \cite{sga}, Expos\'e 1) that the category of
$B[T,T^{-1}]$-comodules is equivalent to the category of $\mathbb Z$-graded
$B$-modules. Thus $N$ has a $\mathbb Z$-grading $N=\oplus_{i\in \mathbb
Z}N_i$, and the comultiplication $\Delta_N:N \to B[T,T^{-1}]\otimes_B N$ is
given by $\sum_{i \in \mathbb Z} x_i \mapsto \sum_{i \in \mathbb Z} T^i
\otimes x_i$. Since $f$ is a nonzerodivisor on $M$, the canonical map $M \to
B\otimes_A M=N$ identifies $M$ with an $A$-submodule of $N$. 

We define our functor $F$ by $FM:=(N,M)$. For this to make sense we must
check that $C_nM \subset M$ for all $n\ge 0$. Let $m \in M$, and write
$m=\sum_{i \in \mathbb Z}x_i$ in $\oplus_{i \in \mathbb Z}N_i=N$. Since the
comodule $N$ was obtained from $M$ by extension of scalars, the element
$x=\Delta_Nm=\sum_{i \in\mathbb Z}T^i\otimes x_i \in B[T,T^{-1}]\otimes_B N$
lies in the image of $S\otimes_A M \to B[T,T^{-1}]\otimes_B N$. Lemma
\ref{lem.sm} then implies that $\mathbf L_n(x)=\sum_{i \in \mathbb
Z}f^n\binom{i}{n}x_i=C_n(m)$ lies in $M$, as desired. 

\begin{theorem}\label{thm.main}
The functor $F:\Rep(G)_f \to \mathcal C$ is an equivalence of categories. 
\end{theorem}
\begin{proof}
Let us first show that $F$ is essentially surjective. Let $(N,M)$ be an
object in $\mathcal C$. We are going to use the comultiplication $\Delta_N:N
\to B[T,T^{-1}]\otimes_B N$ to turn $M$ into an $S$-comodule.

Since $M$ is an $A$-submodule of $N$, it is clear that $f$ is a
nonzerodivisor on $M$. As we have seen before, it follows that $f$ is a
nonzerodivisor on $S\otimes_A M$ and hence that the natural map $S\otimes_A
M \to B\otimes_A(S\otimes_AM)=B[T,T^{-1}]\otimes_B N$ identifies
$S\otimes_AM$ with an $A$-submodule of $B[T,T^{-1}]\otimes_B N$. 

Using Lemma \ref{lem.sm}, we see that our assumption that $C_nM \subset M$
for all $n\ge 0$ is simply the statement that $\Delta_N M\subset S\otimes_A
M$. In other words there exists a unique $A$-module map $\Delta_M:M \to
S\otimes_A M$ such that $\Delta_M$ yields $\Delta_N$ after extending scalars
from $A$ to $B$. 

We claim that $\Delta_M$ makes $M$ into an $S$-comodule. For this we must
check the commutativity of two diagrams, and this follows from the
commutativity of these diagrams after extending scalars from $A$ to $B$,
once one notes that for any two $A$-modules $M_1,M_2$ on which $f$ is a
nonzerodivisor 
\begin{equation}\label{eq.star}
\Hom_A(M_1,M_2)=\{\phi \in \Hom_B(B\otimes_A M_1,B\otimes_A M_2)\,:\,
\phi(M_1)\subset M_2\}.
\end{equation}
Here of course we are identifying $M_1,M_2$ with $A$-submodules of
$B\otimes_A M_1$,$B\otimes_A M_2$ respectively. (At one point we need that
$f$ is a nonzerodivisor on $S\otimes_AS\otimes_AM$, which is true since
$S\otimes_A S$ is flat over $A$.)

As $F$ takes $M$ to $(N,M)$, we are done with essential surjectivity. It is
easy to see that $F$ is fully faithful; this too uses \eqref{eq.star}.
\end{proof} 

\subsection{Principal ideal domains $A$} \label{sec.pid}
One defect of the theorem we just
proved is that it only describes those $G$-modules on which $f$ is a
nonzerodivisor. When $A$ is a principal ideal domain, as we assume for the
rest of this subsection, we can do better. Now $f$ is simply any non-zero
element of $A$. As a consequence of our theorem we obtain an equivalence of
categories between the category $\Rep(G)_{flat}$ of $G$-modules $M$ such
that $M$ is flat as $A$-module and the full subcategory of $\mathcal C$
consisting of pairs $(N,M)$ for which $M$ is a flat $A$-module (in which
case $N\cong B\otimes_A M$ is necessarily a flat $B$-module). 

The next lemma is a variant of \cite[Prop. 3]{serre}.
\begin{lemma}
Let $A$ be a principal ideal domain, let $C$ be a flat $A$-coalgebra, and
let $E$ be a $C$-comodule. Then there exists a short exact sequence of
$C$-comodules
\[
0 \to F_1 \to F_0 \to E \to 0
\] 
in which $F_0$,$F_1$ are flat as $A$-modules. 
\end{lemma}
\begin{proof}
We imitate Serre's proof. Recall (see 1.2 in \emph{loc.\ cit.}) that for any
$A$-module $M$ the map $\Delta\otimes id_M:C\otimes_A M\to C\otimes_A C
\otimes_A M$ ($\Delta$ being the comultiplication for $C$) gives $C
\otimes_A M$ the structure of $C$-comodule, and that (see 1.4 in 
\emph{loc.\ cit.}) the comultiplication map $\Delta_E:E \to C\otimes_A E$ is
an injective comodule map when $C\otimes_A E$ is given the comodule
structure just described. We use $\Delta_E$ to identify $E$ with a
subcomodule of $C\otimes_A E$. 

Now choose a surjective $A$-linear map $p:F \to E$, where $F$ is a free
$A$-module. Let $F_0$ denote the preimage of $E$  under the surjective
comodule map $id\otimes p:C\otimes_A F \twoheadrightarrow C\otimes_A E$.
Since $F_0$ is the kernel of 
\[
C\otimes_A F \to C \otimes_A E \to (C\otimes_ A E)/E,
\] 
it is a subcomodule of $C\otimes_A F$, and $id\otimes p$ restricts to a
surjective comodule map $F_0 \to E$, whose kernel we denote by $F_1$. Since
$C$ and $F$ are flat, so too are $C\otimes_A F$, $F_0$, and $F_1$, and we are
done. We used  that for principal ideal domains, a module is flat
if and only if it is torsion-free, and being torsion-free is a property that
is inherited by submodules. 
\end{proof} 

Returning to our Hopf algebra $S$, we see that any $G$-module $E$ has a
resolution $0\to F_1\to F_0 \to E \to 0$ in which $F_1$,$F_0$ are objects of
$\Rep(G)_{flat}$ and hence are described by our theorem. We conclude that 
 $E$ has the following form. There exist an injective
homomorphism $\phi:N \to N'$ of graded $B$-modules and flat $A$-submodules
$M$,$M'$ of
$N$,$N'$ respectively such that $\phi M \subset M'$ and $(N,M),(N',M') \in
\mathcal C$, having the property that $E$ is isomorphic to $M'/\phi M$ as
$G$-module. 

\subsection{A special case} \label{sec.asf}
When $A$ is a $\mathbb Q$-algebra, the category
$\mathcal C$ is very simple. Indeed, there is a polynomial $P_n \in
\mathbb Q[U]$ of degree $n$ such that $\binom{i}{n}=P_n(i)$, and therefore
$C_n=Q_n(C)$, where $C=C_1$ and $Q_n:=f^nP_n(f^{-1}U) \in A[U]$. Therefore
$\mathcal C$ is the category of pairs $(N,M)$ consisting of a $\mathbb
Z$-graded $B$-module
$N$ and an $A$-submodule $M$ of $N$ such that the natural map $B\otimes_A M
\to N$ is an isomorphism and such that $CM \subset M$, where $C$ is the
endomorphism of the graded module $N=\oplus_{i \in\mathbb Z} N_i$ given by
multiplication by $fi$ on $N_i$. 

When $A$ is the formal power series ring $\mathcal O:=\mathbb
C[[\varepsilon]]$, and $f =\varepsilon^k$ (for some non-negative integer $k$)
our constructions yield a group scheme $G$ over $\mathcal O$ such that
$G(\mathcal O)=\{t \in \mathcal O^\times\,:\, t\equiv
1\mod{\varepsilon^k}\}$, and the category of representations of $G$ on free
$\mathcal O$-modules of finite rank is equivalent to the category of pairs
$(V,M)$, where $V$ is a finite dimensional graded vector space over
$F:=\mathbb C((\varepsilon))$ and $M$ is an $\mathcal O$-lattice in $V$ such
that $CM\subset M$, where $C$ is given by multiplication by $i\varepsilon^k$
on the $i$-th graded piece of $V$. It is amusing to note that for fixed
$V$, the space of all 
$M$ satisfying $CM\subset M$ is an affine Springer fiber, which, when all the
non-zero graded pieces of $V$ are one-dimensional, is actually one of the
affine Springer fibers studied at some length in \cite{gkm}, where it was
shown to be paved by affine spaces.  Finally, since
$\mathcal O$ is a principal ideal domain, the results in \ref{sec.pid} give
a concrete description of all $G$-modules. 

\section{Certain Hopf algebras and their comodules}
Throughout this section $A$ is a commutative ring and $B$ is a commutative
algebra such that the canonical homomorphism $B \otimes_A B \to B$
( given by $b_1\otimes b_2 \mapsto b_1b_2$) is an isomorphism. For example
$B$ might be  of the form $S^{-1}A/I$ for some multiplicative subset $S$
of
$A$ and some ideal $I$ in $S^{-1}A$.

Let $N$ be a $B$-module. Then the canonical $B$-module map $B\otimes_A N \to
N$ (given by $b\otimes n \mapsto bn$) is an isomorphism. It follows that the
canonical $A$-module homomorphism $N \to B \otimes_A N$ (given by $n \mapsto
1\otimes n$) is actually an isomorphism of $B$-modules (since $N\to
B\otimes_A N \to N$ is the identity). 

Moreover, for any two $B$-modules $N_1,N_2$, we have isomorphisms 
\begin{equation}\label{eq.hom}
\Hom_B(N_1,N_2) \cong \Hom_A(N_1,N_2)
\end{equation}
and 
\begin{equation}\label{eq.tensor}
N_1\otimes_A N_2 \cong N_1\otimes_B N_2.
\end{equation}

\subsection{General remarks on Hopf algebras and their comodules} Let $S$ be
a Hopf algebra over $A$. The composition $A\to S \to A$ of the unit and
counit is the identity, and therefore there is a direct sum decomposition
$S=A\oplus S_0$ of $A$-modules, where $S_0$ is by definition the kernel of
the counit $S \to A$. In this subsection all tensor products will be taken
over $A$ and the subscript $A$ will be omitted. 

We denote by $\Delta:S \to S\otimes S$ the comultiplication for $S$. The
counit axioms imply that $\Delta$ takes the form $\Delta(a+s_0)=a+s_0\otimes
1+1\otimes s_0+\bar\Delta(s_0)$, when we identify $S$ with
$A\oplus S_0$ and $S\otimes S$ with $A \oplus (S_0\otimes A) \oplus
(A\otimes S_0) \oplus (S_0\otimes S_0)$. Here $\bar\Delta$ is a uniquely
determined $A$-module map $S_0 \to S_0\otimes S_0$. 

For any $S$-comodule $M$ with comultiplication $\Delta_M:M \to S\otimes M$
the counit axiom for $M$ implies that $\Delta_M(m)=1\otimes m +
\bar\Delta_M(m)$ for a uniquely determined $A$-module map 
\[
\bar\Delta_M:M \to S_0\otimes M.
\] 
In this way we obtain an equivalence of categories between $S$-comodules and
$A$-modules $M$ equipped with an $A$-linear map $\bar\Delta_M:M \to S_0
\otimes M$ such that the diagram 
\begin{equation}\label{cd}
\begin{CD}
M @>\bar\Delta_M>> S_0 \otimes M \\
@V\bar\Delta_MVV @V\bar\Delta\otimes id VV \\
S_0\otimes M @>id\otimes\bar\Delta_M >> S_0\otimes S_0 \otimes M
\end{CD}
\end{equation}
commutes.

\subsection{Hopf algebras for $B$ give Hopf algebras for $A$}  Let $S$ be a
Hopf algebra over $B$. As in the previous subsection we decompose $S$ as
$B\oplus S_0$. It is easy to see that there is a unique Hopf algebra
structure on $R:=A\oplus S_0$ such that the unit and counit for $R$ are the
obvious maps $A \hookrightarrow R$ and $R \twoheadrightarrow A$ and such
that the Hopf algebra structure on $B\otimes_A R$ agrees with the given one
on $S$ under the natural $B$-module isomorphism $B\otimes_A R \cong S$. What
makes this work is \eqref{eq.tensor}, a consequence of our assumption that
$B\otimes_A B \to B$ is an isomorphism, so that, for example, $S_0 \otimes_B
S_0 \cong S_0 \otimes_A S_0$. The comultiplications $\Delta_R$,$\Delta_S$ on
$R$,$S$ respectively are given by 
\begin{align}
\Delta_R(a+s_0)&=a+s_0\otimes
1+1\otimes s_0+\bar\Delta(s_0)\\
\Delta_S(b+s_0)&=b+s_0\otimes
1+1\otimes s_0+\bar\Delta(s_0)
\end{align}
and similar considerations apply to the multiplication maps $R\otimes_A R
\to R$, $S \otimes_B S \to S$ and the antipodes $R \to R$, $S \to S$. 

\begin{proposition}
The category of $R$-comodules is equivalent to the category of $A$-modules
$M$ equipped with an $S$-comodule structure on $N:=B\otimes_A M$. 
\end{proposition} 
\begin{proof}
We have already observed that giving an $R$-comodule is the same as giving
an $A$-module $M$ equipped with an $A$-module map $\bar\Delta_M:M\to S_0
\otimes_A M$ such that \eqref{cd} commutes. Since $S_0$ is a $B$-module and
$B\otimes_A B \cong B$, giving $\bar\Delta_M$ such that \eqref{cd} commutes
is the same as giving a $B$-module map $\bar\Delta_N:N \to S_0\otimes_BN$
such that 
\begin{equation*}
\begin{CD}
N @>\bar\Delta_N>> S_0 \otimes_B N \\
@V\bar\Delta_NVV @V\bar\Delta\otimes id VV \\
S_0\otimes_B N @>id\otimes\bar\Delta_N >> S_0\otimes_B S_0 \otimes_B N
\end{CD}
\end{equation*}
commutes, or, in other words, giving an $S$-comodule structure on $N$. 
\end{proof} 

\subsection{Special case}\label{sec.dvr}  
Let $\mathcal O$ be a valuation ring and $F$ its
field of fractions. Let $G$ be an affine group scheme over $F$ and let $S$ be
the corresponding commutative Hopf algebra over $F$. Decompose $S$ as $F
\oplus S_0$ and define a commutative Hopf algebra $R$ over $\mathcal O$ by
$R:=\mathcal O \oplus S_0$. Corresponding to $R$ is an affine group scheme
$\tilde G$ over $\mathcal O$, and giving a representation of $\tilde G$
(that is, an $R$-comodule) is the same as giving an $\mathcal O$-module $M$
together with an $S$-comodule structure on $F\otimes_\mathcal O M$. 

For example, when $G$ is the multiplicative group $\mathbb G_m$, the Hopf
algebra $R$ is $\{\sum_{i \in \mathbb Z} a_iT^i \in F[T,T^{-1}]\,:\,\sum_{i
\in \mathbb Z} a_i \in \mathcal O\}$, which is easily seen to be the union
of the Hopf subalgebras $S_k$ discussed in the introduction.

\end{document}